\documentclass[a4paper,twoside,12pt]{amsart}
\usepackage{amsmath}
\usepackage{amsthm,color}
\usepackage{amssymb}
\usepackage{amsmath}
\usepackage{amssymb}
\usepackage{ifpdf}
\ifpdf
  \usepackage[colorlinks,linkcolor=red,anchorcolor=blue,citecolor=green]{hyperref}
\else
\fi

\newtheorem{theorem}{Theorem}[section]

\newtheorem{corollary}[theorem]{Corollary}

\newtheorem{lemma}[theorem]{Lemma}

\newtheorem{proposition}[theorem]{Proposition}
\newtheorem{question}[theorem]{Question}

\numberwithin{equation}{section}

\newcommand{\ACA}{\mathsf{ACA}_0}
\newcommand{\RCA}{\mathsf{RCA}_0}
\newcommand{\WKL}{\mathsf{WKL}_0}
\newcommand{\WWKL}{\mathsf{WWKL}_0}
\newcommand{\TT}{\mathsf{TT}}

\begin{document}

\title[Perfect Subsets]{On The Computability of Perfect Subsets of Sets with Positive Measure}


\author{C.~T.~Chong}
\address{Department of Mathematics\\
National University of Singapore\\Singapore 119076}
\email{chongct@math.nus.edu.sg}

\author{Wei Li}
\address{Department of Mathematics\\National University of Singapore\\Singapore 119076}
\email{matliw@nus.edu.sg}

\author{Wei Wang}
\address{Institute of Logic and Cognition and Department of Philosophy\\Sun Yat-Sen University\\Guangzhou, China}
\email{wwang.cn@gmail.com}

\author{Yue Yang}
\address{Department of Mathematics\\National University of Singapore\\Singapore 119076}
\email{matyangy@nus.edu.sg}

\thanks{The authors thank many logicians, in particular Steffen Lempp and Ludovic Patey for inspiring discussions and suggestions, and the referee for suggesting a simpler proof of Proposition 2.4 and many other helpful comments. Chong's research was partially supported by NUS grants C-146-000-042-001 and WBS : R389-000-040-101.
Wang's research was partially supported by China NSF Grant 11471342.
Yang's research was partially supported by NUS AcRF Tier 1 grant R146-000-231-114 and MOE2016-T2-1-019.
All the authors acknowledge the support of JSPS-NUS grants R-146-000-192-133 and R-146-000-192-733 during the course of the work.}

\begin{abstract}
A set $X \subseteq 2^\omega$ with positive measure contains a perfect subset. We study such perfect subsets from the viewpoint of computability and prove that these sets can have weak computational strength. Then we connect the existence of perfect subsets of sets with positive measure with reverse mathematics.
\end{abstract}

\maketitle

\section{Introduction}

An observation made in algorithmic randomness is that a set with positive measure (a \emph{positive} set for short) in Cantor space contains members with limited computational strength, and in many cases those with limited computational strength form sets of measure $1$. As examples, almost no member of a positive set computes the halting problem or a given non-computable set, and the class of generalized low reals is of measure $1$. There are numerous such examples in algorithmic randomness. This paper is motivated by the general question: to what extent does the above observation hold for perfect subsets of a positive set?

Clearly a positive set contains many perfect subsets.  We identify a perfect set $C$ with a perfect (binary) tree
$$
    P = \{\sigma \in 2^{<\omega}: \sigma \text{ is an initial segment of some } X \in C\}.
$$
If $P$ is a perfect tree then it computes its own \emph{growth rate}
$$
    g_P: k \mapsto \min\{m: |P \cap 2^m| \geq 2^k\}.
$$
Since a perfect tree contains perfect subtrees of arbitrarily fast growth rate, every positive set contains perfect subsets which compute the halting problem. Hence in terms of computational strength it appears that the situation with perfect subsets is in sharp contrast to  that with members  of a positive set. Nevertheless, in Section \ref{s:computability} we demonstrate two computability-theoretic properties of members of a positive set which are shared by  perfect subsets of a positive set. The first is an analogue of the Low Basis Theorem for positive $\Pi^0_1$ classes. The second asserts that every positive set contains perfect subsets not computing a given non-computable set. We also include in Section \ref{s:computability} a proposition showing that the above results do not hold for effectively closed null sets. The statement that every positive set contains a perfect subset turns out to be interesting also from  the perspective of reverse mathematics, as we argue in Section \ref{s:RevMath}.

We fix the following  notations and  terminologies. If $T \subseteq 2^{<\omega}$ is a tree and $\sigma \in 2^{<\omega}$ then
$$
    T_\sigma = \{\tau \in T: \tau \text{ is comparable with } \sigma\}.
$$
The set $[T]$ of a tree $T$ is the set of all infinite paths on $T$. A binary tree $T$ is \emph{positive} iff $[T]$ is positive. An \emph{initial segment} of a tree $T$ is a tree $F$ such that  every node of $T$ is either a node of $F$ or an extension of a leaf of $F$. $T$ \emph{end-extends} $F$ if $F$ is an initial segment of $T$.

For other notions and notations, the reader is referred to Downey and Hirschfeldt  \cite{Downey.Hirschfeldt:2010.book}  for computability and algorithmic randomness, and to Simpson's monograph \cite{Simpson:2009.SOSOA} for reverse mathematics.

\section{The Computability of Perfect Subsets of Positive Sets}\label{s:computability}

We begin  with positive $\Pi^0_1$ classes. As usual, a $\Pi^0_1$ class is identified with the collection of infinite paths on a computable binary tree. Since there exist positive computable binary trees containing no computable paths, there exist positive computable trees with no computable perfect subtree. So the following analogue of the Low Basis Theorem for positive $\Pi^0_1$ classes is non-trivial. The idea is to define a tree of perfect subtrees of a computable positive tree and then apply the Low Basis Theorem. We need the following form of Lemma 8 in Ku\v{c}era \cite{Kucera:85}, which can be obtained in a way similar to \cite[Lemma 8.5.2]{Downey.Hirschfeldt:2010.book}.

\begin{lemma}[Ku\v{c}era]\label{lem:Kucera-coding}
For every positive tree $T$ and every $\epsilon > 0$ there exist an infinite subtree $S_\epsilon \subseteq T$ and a computable function $\rho_\epsilon: 2^{<\omega} \to \mathbb{Q}^+$ such that $S_\epsilon \leq_T T$, $\mu([S_\epsilon]) > \mu([T]) - \epsilon$ and
$$
    [S_{\epsilon, \sigma}] \neq \emptyset \text{ if and only if } \mu([S_{\epsilon, \sigma}]) > \rho_\epsilon(\sigma)
$$
for every $\sigma \in 2^{<\omega}$.
\end{lemma}

\begin{proposition}\label{prp:low-perfect-subtree}
If $T$ is a computable positive tree and $\epsilon > 0$ then there is a low perfect subtree $P$ of $T$ with $\mu([P]) > \mu([T]) - \epsilon$.
\end{proposition}

\begin{proof}
In Lemma \ref{lem:Kucera-coding}, let $S = S_{\epsilon/2}$ and $\rho=\rho_{\epsilon/2}$ be the associated tree and function respectively. 

We define a  computable function $g: \omega \to \omega$ by induction as follows. Let $g(0) = 0$.
Suppose that $g(n) = k$ is defined. Let $r(n) = \min \{\rho(\sigma): |\sigma| \leq k\}$. By the assumption on $\rho$, $r(n) > 0$. Let $g(n+1)$ be the least $l > k$ such that $2^{-l} < r(n)$.

Fix a positive rational $\delta$ between $\mu([S]) - \epsilon/2$ and $1$. For each finite tree $U$, let $\|U\| = \max \{|\sigma|+1: \sigma \in U\}$. Let $\mathcal{F}$ be the set of finite trees $U \subset S$ satisfying the following conditions:
\begin{enumerate}
  \item if $n < \|U\|$ then $|U \cap 2^n| > \delta 2^n$;
  \item if $g(i+1) < \|U\|$ and $\sigma \in U \cap 2^{g(i)}$ then $\sigma$ has two distinct extensions in $U \cap 2^{g(i+1)}$.
\end{enumerate}
As $g$ is computable, $\mathcal{F}$ can be identified as a computably bounded computable tree. Moreover, every infinite sequence $X$ in $\mathcal{F}$ gives rise to a perfect subtree $\bigcup X \subseteq S$ with $\mu([\bigcup X]) \geq \delta$. We show that $\mathcal{F}$ is infinite.

For each $n$, let
$$
  U_n = \{\sigma \in S \cap 2^{\leq g(n+1)}: [S_\sigma] \neq \emptyset\} = \{\sigma \in S \cap 2^{\leq g(n+1)}: \mu([S_\sigma]) > \rho(\sigma)\}.
$$
As $\mu([S]) > \delta$, $U_n$ satisfies (1) above. Suppose that $g(i+1) < \|U_n\|$ and $\sigma \in U_n \cap 2^{g(i)}$. Then $\mu([S_\sigma]) > 2^{-g(i+1)}$. So $\sigma$ has two distinct extensions $\tau_0$ and $\tau_1$ in $S \cap 2^{g(i+1)}$ with both $S_{\tau_0}$ and $S_{\tau_1}$ positive. Hence $\sigma$ has two distinct extensions in $U_n \cap 2^{g(i+1)}$ and $U_n$ satisfies (2). This proves that $U_n \in \mathcal{F}$ and thus $\mathcal{F}$ is infinite.

By the Low Basis Theorem, $[\mathcal{F}]$ contains a low path $X$, and $\bigcup X$ is a low perfect subtree of $T$ as desired.
\end{proof}

The proof of the above proposition leads to the following corollary.

\begin{corollary}
Let $T \subseteq 2^{<\omega}$ be a positive tree.
\begin{enumerate}
    \item There exists an infinite tree $\mathcal{F}$ such that $\mathcal{F}$ is $T$-computable and computably bounded and $\bigcup X$ is a positive perfect subtree of $T$ for every $X \in [\mathcal{F}]$.
    \item Every set computing $T'$ is Turing equivalent to $T \oplus P$ for some positive perfect subtree $P \subseteq T$.
\end{enumerate}
\end{corollary}

\begin{proof}
(1) is essentially a part of the proof of Proposition \ref{prp:low-perfect-subtree}. Note that $\mathcal{F}$ is computably bounded instead of $T$-computably bounded, because $\rho_\epsilon$ in Lemma \ref{lem:Kucera-coding} is computable and so is the function $g$ defined in the proof of the above proposition. (2) can be obtained by combining the above proof with that of \cite[Theorem 7]{Kucera:85} (see also \cite[Theorem 8.5.1]{Downey.Hirschfeldt:2010.book}).
\end{proof}

However,  the conclusion of Proposition \ref{prp:low-perfect-subtree} does not hold for an arbitrary $\Pi^0_1$ class.

\begin{proposition}\label{prp:Pi1-perfect-subtree}
There exists a computable binary tree $T$ such that $[T]$ is a perfect set and every perfect subtree of $T$ computes the halting problem.
\end{proposition}

\begin{proof}
Let $m$ be the modulus function of $\emptyset'$, i.e.,
$$
    m(x) = \min \{s: \forall e < x(\Phi_e(e) \downarrow \to \Phi_e(e)[s] \downarrow)\}.
$$
For each $X \in 2^{\omega}$, let $p^X(n)$ be the $n$-th $i$ such that $X(i) = 1$. So $p^X$ is total iff $X(i) = 1$ for infinitely many $i$. Let
$$
    C = \{X \in 2^\omega: \forall n (p^X(n) \geq m(n))\}.
$$
So $C$ is a perfect $\Pi^0_1$-class and every $X \in C$ with $p^X$ total computes the halting problem. Let $T$ be a computable binary tree such that $C = [T]$. If $P$ is a perfect subtree of $T$, then $P$ computes some $X \in [P] \subseteq C$ with $p^X$ total and $X$ in turn computes the halting problem. Hence $P$ computes the halting problem as well.
\end{proof}

Next, we move to perfect subtrees of arbitrary positive trees.

\begin{theorem}\label{thm:perfect-subtree-cone-avoidance}
Every positive tree has a perfect subtree not computing a given non-computable set.
\end{theorem}

\begin{proof}
Fix a binary tree $\hat{T}$ with $\mu([\hat{T}]) > 0$ and a non-computable $X$.

We build a desired perfect subtree of $\hat{T}$ by a variant of Mathias forcing. A forcing condition is a pair $(F, T)$ such that $F$ is isomorphic to some $2^{<n}$, $X \not\leq_T T$ and
\begin{equation}\label{eq:psca-forcing-condition}
    \mu([(T \cap \hat{T})_\sigma]) > (1 - 2^{-2}) 2^{-|\sigma|}
\end{equation}
for every leaf $\sigma$ of $F$. An extension of a condition $(F,T)$ is a condition $(E, S)$ such that $F$ is an initial segment of $E$ and $S \subseteq T$. As $\mu([\hat{T}]) > 0$, there exists $\sigma$ with $\mu([\hat{T}_\sigma]) > (1 - 2^{-2}) 2^{-|\sigma|}$. Let $F = \{\sigma \upharpoonright l: l \leq |\sigma|\}$. Then $(F, 2^{<\omega})$ is a condition. Each condition $(F,T)$ represents the following set of perfect subtrees of $\hat{T}$:
$$
    [F, T] = \{P \subseteq \hat{T} \cap T: P \text{ is a perfect tree end-extending } F\}.
$$

Firstly, we show that for each $e$ the conditions forcing $\Phi_e(P) \neq X$ are dense. Fix a condition $(F,T)$ and some $e$, let $\mathcal{U}$ be the set of binary trees $S$ such that $S$ satisfies \eqref{eq:psca-forcing-condition} in  place of $\hat{T}$ and
$$
    (\forall x)((\forall i < 2)(E_i \text{ end-extends } F \wedge \Phi_e(E_i; x) \downarrow) \to \Phi_e(E_0; x) = \Phi_e(E_1; x))
$$
for  all pairs $(E_0,E_1)$ of finite subtrees of $S$. Note that $\mathcal{U}$ can be identified with a $\Pi^T_1$ class in Cantor space.

\begin{lemma}\label{lem:psca-forcing-Pi1}
Suppose that $\mathcal{U} \neq \emptyset$. Then there exists a condition $(E,S)$ extending $(F,T)$ such that $\Phi_e(P) \neq X$ whenever $P \in [E,S]$.
\end{lemma}

\begin{proof}
As $\mathcal{U}$ is $\Pi^T_1$ and $X \not\leq_T T$, we can apply the cone avoidance property of $\Pi^0_1$ classes relativized to $T$ and get an infinite binary tree $S_0 \in \mathcal{U}$ s.t. $X$ is not computable in $T \oplus S_0$. Let $S = T \cap S_0$. By the definition of $\mathcal{U}$, for each finite subtree $E$ of $S$ which end-extends $F$, the value of $\Phi_e(E; x)$ is independent of $E$ if the computation halts. Thus, if $P$ is a perfect subtree of $S$ end-extending $F$ and $\Phi_e(P)$ is total then $\Phi_e(P)$ is computable in $T \oplus S_0$ and thus cannot be equal to $X$. By \eqref{eq:psca-forcing-condition} and the version of \eqref{eq:psca-forcing-condition} for $S_0$,
$$
    \mu([(\hat{T} \cap S)_\sigma]) > (1 - 2^{-1}) 2^{-|\sigma|}
$$
for every leaf $\sigma$ of $F$. By the Lebesgue Density Theorem, we can extend the leaves  of $F$ to obtain a finite tree $E$ such that $E$ is isomorphic to $F$ and
$$
    \mu([(\hat{T} \cap S)_\tau]) > (1 - 2^{-2}) 2^{-|\tau|}
$$
for every leaf $\tau$ of $E$. Clearly, $(E,S)$ satisfies the conclusion.
\end{proof}

\begin{lemma}\label{lem:psca-forcing-Sigma1}
Suppose that $\mathcal{U} = \emptyset$. Then there exists $E$ such that $(E, T)$ is an extension of $(F,T)$ and $\Phi_e(E; n) \downarrow \neq X(n)$ for some $n$.
\end{lemma}

\begin{proof}
Let
\begin{equation}\label{eq:psca-tildeT}
    \tilde{T} = \{\sigma: \mu([(T \cap \hat{T})_\sigma]) > 0\} = T \cap \hat{T} \setminus  \{\sigma: \mu([(T \cap \hat{T})_\sigma]) = 0\}.
\end{equation}
Then $\tilde{T}$ is a binary tree satisfying \eqref{eq:psca-forcing-condition} in  place of $\hat{T}$. As $\tilde{T} \not\in \mathcal{U}$, there exist $n$ and a finite subtree $D$ of $\tilde{T}$ such that $D$ is an end-extension of $F$ and $\Phi_e(D; n) \downarrow \neq X(n)$. By \eqref{eq:psca-tildeT}, we may assume that $D$ is isomorphic to some $2^{<l}$. By the Lebesgue Density Theorem and \eqref{eq:psca-tildeT}, we have a finite end-extension $E$ of $D$ obtained  by extending each leaf $\sigma$ of $D$ to some $\tau$ such that $\mu([(T \cap \hat{T})_\tau]) > (1-2^{-2}) 2^{-|\tau|}$. Then $(E,T)$ is the desired extension of $(F,T)$.
\end{proof}

The conditions $(F,T)$ with  splitting extensions  on $F$ are also dense:

\begin{lemma}\label{lem:psca-forcing-infinity}
Every condition $(E,T)$ with $E$ isomorphic to $2^{<n}$ has an extension $(F,T)$ with $F$ isomorphic to $2^{<n+1}$.
\end{lemma}

\begin{proof}
Suppose that $(E,T)$ is a condition with $E$ isomorphic to $2^{<n}$. For each leaf $\sigma$ of $E$, by the Lebesgue Density Theorem and \eqref{eq:psca-forcing-condition} choose
 two incomparable $\tau(\sigma,0)$ and $\tau(\sigma,1)$ extending $\sigma$ such that
$$
    \mu([(T \cap \hat{T})_{\tau(\sigma,i)}]) > (1 - 2^{-2}) 2^{-|\tau(\sigma,i)|}
$$
for $i < 2$. Let
$$
    F = \{\eta: \eta \text{ is an initial segment of some } \tau(\sigma,i), i < 2, \sigma \text{ is a leaf of } E\}.
$$
Then $(F,T)$ is as desired.
\end{proof}

Theorem \ref{thm:perfect-subtree-cone-avoidance} follows from Lemmata \ref{lem:psca-forcing-Pi1}, \ref{lem:psca-forcing-Sigma1} and \ref{lem:psca-forcing-infinity}.
\end{proof}


\begin{corollary}
Given a non-computable $X$, every positive set has a perfect subset not computing $X$.
\end{corollary}

\begin{proof}
It follows from Theorem \ref{thm:perfect-subtree-cone-avoidance} and that every positive set has a positive closed subset.
\end{proof}

\section{Perfect Sets and Reverse Mathematics}\label{s:RevMath}

As observed by Lempp, the statement ``every positive closed set contains a perfect subset'' can be formulated as a $\Pi^1_2$ sentence in second order arithmetic. In general, if $\Lambda$ is a pointclass then we may consider the following statement:
\begin{equation}\label{eq:positive-perfect}
    \text{every positive set in } \Lambda \text{ contains a perfect subset.}
\end{equation}
Clearly, the proposition that every positive closed set contains a perfect subset implies $\WWKL$ over $\RCA$, and it is a consequence of $\WKL$ by the proof of Proposition \ref{prp:low-perfect-subtree}. So it is natural to ask whether \eqref{eq:positive-perfect} for closed sets is a substantially new principle. During a discussion with the third author, Ludovic Patey proved that this instance for closed set does not imply $\WKL$ over $\RCA$, by combining ideas from the proof of Theorem \ref{thm:perfect-subtree-cone-avoidance} and Liu \cite{Liu:2012}. But whether $\WWKL$ implies \eqref{eq:positive-perfect} for closed sets remains open.

Here, we present another connection between the \eqref{eq:positive-perfect} family and reverse mathematics. Recall that $\TT^1$ is the proposition that for every finite coloring $f: 2^{<\omega} \to n$ there exists an $f$-homogeneous perfect tree $T$.

\begin{proposition}[$\RCA$]\label{prp:P-TT1}
The instance of \eqref{eq:positive-perfect} for $\Pi^0_2$ sets implies $\TT^1$.
\end{proposition}

\begin{proof}
It is obvious that over $\RCA$ the instance of \eqref{eq:positive-perfect} for $\Pi^0_2$ sets implies $2\mathsf{-POS}$, which is the statement introduced by Avigad et al. \cite{Avigad.Dean.ea:2012} that every $\Pi^0_2$ set with positive outer measure is non-empty. By \cite[Theorem 3.7 ]{Avigad.Dean.ea:2012}, $\RCA + 2\mathsf{-POS} \vdash B\Sigma^0_2$.

Let $\mathcal{M}$ be a model of $\RCA$ and the instance of \eqref{eq:positive-perfect} for $\Pi^0_2$ sets. Let $f \in \mathcal{M}$ be a coloring of the full binary tree with $a \in \mathcal{M}$ many colors. Since $\mathcal{M} \models B\Sigma^0_2$ by the above paragraph, there exists $i < a$ such that the following $\Pi^{0,f}_2$ set has outer measure at least $1/a$:
$$
    \mathcal{A} = \{X: \forall m \exists n > m(f(X \upharpoonright n) = i)\}.
$$
An application of \eqref{eq:positive-perfect} to $\mathcal{A}$ produces a perfect tree $T \in \mathcal{M}$ such that $[T] \subseteq \mathcal{A}$. So
$$
    \forall \sigma \in T \exists \tau \in T(\sigma \text{ is an initial segment of } \tau \wedge f(\tau) = i).
$$
By $I\Sigma^0_1$ in $\mathcal{M}$, there exists $P \in \mathcal{M}$ which is computable in $T$ and $f$-homogeneous.
\end{proof}

Corduan et al. \cite{Corduan.Groszek.ea:2010.TT1} prove that $\TT^1$ is strictly stronger than $B\Sigma^0_2$ over $\RCA$, and recently the authors \cite{Chong.Li.ea:2018.TT1} prove that $\TT^1$ is strictly weaker than $I\Sigma^0_2$ over $\RCA$. But the exact first order theory of $\TT^1$ remains unknown. By Proposition \ref{prp:P-TT1}, the instance of \eqref{eq:positive-perfect} for $\Pi^0_2$ sets can be regarded as a natural strengthening of $\TT^1$, hence it could be an interesting subject as well as $\TT^1$. At the moment, we do not know much more about this instance, except the following immediate corollary of Theorem \ref{thm:perfect-subtree-cone-avoidance}.

\begin{corollary}[$\RCA$]\label{cor:projective-P}
The schema of \eqref{eq:positive-perfect} for definable sets (equivalently the schema of \eqref{eq:positive-perfect} for projective sets) is strictly weaker than $\ACA$.
\end{corollary}

A $\Pi^{0,X'}_1$ set is also a $\Pi^{0,X}_2$ set. By \cite[Proposition 3.4]{Avigad.Dean.ea:2012}, over $\RCA + B\Sigma^0_2$ every positive $\Pi^{0,X}_2$ set contains a positive $\Pi^{0,X'}_1$ set. Moreover, that every positive $\Pi^{0,X'}_1$ set contains a perfect subset where $X$ ranges over all second order elements, trivially implies $\mathsf{2-WWKL}$ and thus also implies $B\Sigma^0_2$ by \cite[Theorem 3.7]{Avigad.Dean.ea:2012}. So,  \eqref{eq:positive-perfect} for $\Pi^0_2$ sets and \eqref{eq:positive-perfect} for all $X$ and all $\Pi^{0,X'}_1$ sets are equivalent over $\RCA$. It is then natural to wish that the proof of Theorem \ref{thm:perfect-subtree-cone-avoidance} could be useful for studying \eqref{eq:positive-perfect} for $\Pi^0_2$ sets in non-standard models. But the current form of the proof depends on certain instances of Lebesgue Density Theorem, which means that the proof does not work in the absence of $I\Sigma^0_2$, by the following observation.

\begin{proposition}[$I\Sigma_1$]\label{prp:Density}
The following statements are equivalent:
\begin{enumerate}
    \item $I\Sigma_2$;
    \item if $T$ is a $\emptyset'$-computable tree with $\mu([T]) > 0$ and $\epsilon > 0$ then there exists a $\sigma$ such that $\mu([T_\sigma]) > (1 - \epsilon) 2^{-|\sigma|}$.
\end{enumerate}
\end{proposition}

\begin{proof}
(2) $\Rightarrow$ (1). Let $M \models I\Sigma_1 + \neg I\Sigma_2$. As $M \not\models I\Sigma_2$, there exist a $\Sigma_2$-cut $I \subset M$ and a $\Sigma_2$-function $g: I \to M$ which is cofinal in $M$. Moreover, there is a uniformly computable family $(g_s: s \in M)$ such that $I = \{i: g(i) = \lim_s g_s(i)\}$. Let $b \in M$ be an upper bound of $I$ and let $T$ be the set of $\sigma \in M \cap 2^{<M}$ such that
$$
    (\forall i < b) ((\forall t > |\sigma|) (g_{|\sigma|}(i) = g_t(i) < |\sigma|) \to \sigma(g_{|\sigma|}(i)) = 1).
$$
By $I\Sigma_1$ (indeed $B\Sigma_1$), the above definition is equivalent to a $\Sigma_1$ formula, so $T$ is a $\emptyset'$-computable tree and $T \cap 2^s \in M$ for every $s \in M$. Hence the following calculation can proceed in $M$ for each $s \in M$, 
$$
    |T \cap 2^s| \geq (1 - \sum_{k < |F_s|} 2^{-k-1}) 2^s > 2^{s-b},
$$
where $F_s = \{i < b: \forall t > s (g_s(i) = g_t(i) < s)\} \in M$. Hence $\mu([T]) > 0$. But from the definition of $T$ and the cofinality of $g$, $\mu([T_\sigma]) \leq 2^{-|\sigma|-1}$ for any $\sigma$. So (2) implies (1).

(1) $\Rightarrow$ (2). Now let $M \models I\Sigma_2$ and $T$ be a $\emptyset'$-computable tree with $\mu([T]) > \delta > 0$. We work in $M$. Fix $\epsilon > 0$. Pick $n$ such that $2^{-n} < \epsilon \delta$. By $I\Sigma_2$, let
$$
    k = \max\{j \leq 2^n: \exists l (|2^l - T| \geq j 2^{l-n})\}
$$
and let $l$ be such that $|2^l - T| \geq k 2^{l-n}$. As $\mu([T]) > \delta$, $|2^l \cap T| > \delta 2^l$ and $k < 2^n$. We claim that $\mu([T_\sigma]) > (1 - \epsilon) 2^{-|\sigma|}$ for some $\sigma \in 2^l \cap T$. For if otherwise, each $\sigma \in 2^l \cap T$ corresponds to some $l_\sigma$ such that
$$
    |T_\sigma \cap 2^{l+l_\sigma}| \leq (1-\epsilon) 2^{l_\sigma}.
$$
By $B\Sigma_2$, we can find a common upper bound $m$ of all these $l_\sigma$'s. Then
$$
    |2^{l+m} - T| > k 2^{-n} + \epsilon \delta > (k+1) 2^{-n},
$$
contradicting the maximality of $k$. So (1) implies (2).
\end{proof}

\section{Questions}\label{s:questions}

We conclude this article with some questions, including those mentioned in the previous sections.

\begin{question}
Does $\WWKL$ imply that every positive closed set contains a perfect subset?
\end{question}

A related computability question is as follows.

\begin{question}
Is there a (computable or not) tree $T$ such that $\mu([T]) > 0$ but the oracles computing a perfect subtree of $T$ form a null set?
\end{question}

In light of Section \ref{s:RevMath}, we could ask many questions about \eqref{eq:positive-perfect} for $\Pi^0_2$ sets, in particular the specific one below.

\begin{question}
Does \eqref{eq:positive-perfect} for $\Pi^0_2$ sets imply $I\Sigma^0_2$ over $\RCA$?
\end{question}

Unlike Proposition \ref{prp:low-perfect-subtree}, the perfect subtrees obtained in Theorem \ref{thm:perfect-subtree-cone-avoidance} are not positive. So we may raise the following question.

\begin{question}
Does every positive tree contain a positive perfect subtree which does not compute some fixed non-computable set?
\end{question}

\bibliographystyle{plain}

\end{document}